\documentclass{article}
\usepackage{amsmath,amssymb,amscd,amsthm,amsfonts,amstext,amsxtra,
}
\usepackage{eufrak,fontenc,fontsmpl,latexsym,rawfonts}
\theoremstyle{plain}
\newtheorem{theorem}{Theorem}

\newtheorem{proposition}[theorem]{Proposition}

\newtheorem*{theorem*}{Theorem}
\newtheorem*{lemma*}{Lemma}
\newtheorem*{proposition*}{Proposition}
\newtheorem*{corollary*}{Corollary}
\newtheorem*{definition*}{Definition}
\newtheorem*{conjecture*}{Conjecture}
\theoremstyle{definition}
\newtheorem*{remark*}{Remark}

\newtheorem*{keywords}{Keywords}

\begin{document}
\title{Direct image for multiplicative and relative
$K$-theories from transgression of the families index theorem, part 3}
\author{Alain Berthomieu\\\and Universit\'e de Toulouse, C.U.F.R. J.-F. Champollion\and
\!\!\!\!\!\!\!\!I.M.T. (Institut de Math\'ematiques de Toulouse, UMR CNRS n$^\circ$ 5219)\and Campus d'Albi, Place de Verdun,
81012 Albi Cedex, France.
\and{\tt{alain.berthomieu@univ-jfc.fr}}}
\maketitle
\begin{abstract}
This is the final part of the work started in \cite{MoiPartie1}
and \cite{MoiPartie2}, (the ``longstanding forthcoming preprint'' referred to in \cite{MoiOberwolfach}).

Here the question of double fibration is adressed both for relative $K$-theory
and ``free multiplicative'' $K$-theory.

In the case of relative and ``nonfree'' multiplicative $K$-theory, the direct image is proved to be functorial for double submersions.
\begin{keywords} multiplicative $K$-theory, families index, Chern-Simons transgression, proper submersions.
\end{keywords}
\noindent{\textbf{AMS-classification:}} Primary: 14F05, 19E20, 57R20, secondary: 14F40, 19D55, 53C05, 55R50.
\end{abstract}
\section{Introduction:}
This is the final part of the series started in \cite{MoiPartie1} and continued in \cite{MoiPartie2}.

Here, double fibrations are studied, and the functoriality of the direct image
constructed in \cite{MoiPartie1} for relative $K$-theory is
established, while in the case of ``free multiplicative'' $K$-theory,
the functoriality of the direct image constructed in \cite{MoiPartie2}
is proved when restricted to ``nonfree'' multiplicative $K$-theory.

Meanwhile, the work \cite{BunkeSchick} was diffused. The relations
between \cite{BunkeSchick} and \cite{MoiPartie1}, \cite{MoiPartie2} and the
present paper will be explained elsewhere.

Consider two submersions $\pi_1\colon M\longrightarrow B$ and $\pi_2\colon B\longrightarrow S$ and the composed submersion
$\pi_2\circ\pi_1\colon M\longrightarrow S$. The goal of this section is to compare
direct image with respect to $\pi_2\circ\pi_1$ (the ``one step'' direct image) and the composition of the two direct images relative to $\pi_1$ and $\pi_2$
(the ``two step'' direct image) for topological, relative and multiplicative $K$-theories.

The paper is organised as follows: in part \ref{topol}, the
construction of a canonical (``topological'') link (see \S2.1.2 in \cite{MoiPartie1}
for the definition of a ``link'') between representatives
of one step and two step direct images in topological $K$-theory
of some same vector bundle on $M$ is constructed. This needs
some analysis of the spectral behaviour of fibral Dirac operators
in the adiabatic limit.
In section \ref{flat}, the Leray spectral sequence is recalled, and
it is explained how a (``sheaf theoretic'') link between one step and two step direct image
of a flat vector bundle can be obtained from this spectral sequence.
The equivalence of the topological and the sheaf theoretic links is then proved.

The functoriality of the direct image for relative $K$-theory is then a direct consequence of this equivalence result (see section \ref{relat}).

Finally, the caracterisation of the form $\tau$ (performed in \S2.7.2 of
\cite{MoiPartie2})
entering in the definition of the direct image for multiplicative $K$-theory
suffices to establish the functoriality for ``nonfree'' multiplicative $K$-theory (see section \ref{mult}).
\section{Topological $K$-theory:}\label{topol}
The goal of this paragraph is to
construct a canonical link between couples of bundles obtained from the same one, by one step and two steps direct image from $K^0_{\rm{top}}(M)$ to $K^0_{\rm{top}}(S)$.
\subsection{Fibral exterior differential:}\label{doublefibrext}
The respective vertical tangent vector bundles  associated with $\pi_1$, $\pi_2$ and $\pi_2\circ\pi_1$ will be denoted by $TM\!/\!B$, $TB\!/\!S$ and $TM\!/\!S$. Put some vector bundle $\xi$ with connexion $\nabla_{\!\xi}$ on $M$.
Call ${\mathcal E}^\pm_{M/S}$ (resp. ${\mathcal E}^\pm_{M/B}$) the infinite rank vector bundles on $S$ (resp. $B$)
of even/odd degree differential forms along the fibres of $\pi_2\circ\pi_1$ (resp. $\pi_1$) with values in $\xi$. 

Choose any ${\mathcal C}^\infty$ supplementary subbundle $T^H\!M$ of
$TM\!/\!B$ in $TM\!/\!S$. Of course $T^H\!M\cong\pi_1^*TB\!/\!S$. On the fibre of $\pi_2\circ\pi_1$ over any point $s$ of $S$, one obtains for $\xi$-valued differential forms an isomorphism
analogous to formula (37) of \cite{MoiPartie1}:
\begin{equation}\label{E:isomabove}
{\mathcal E}_{M/S}\cong\Omega\big(\pi_2^{-1}(\{s\}),{\mathcal E}_{M/B}\big)\end{equation}

For any $b\in B$ and any tangent vector $U\in T_bB\!/\!S$, call $U^H$ its horizontal lift as a section of $T^H\!M$ over $\pi^{-1}_1(b)$.
For any $s\in S$, the construction of formula (53) of \cite{MoiPartie1} produces a connection 
on the restriction of ${\mathcal E}_{M/B}$ over $\pi^{-1}(\{s\})$ which
will be denoted $\overline\nabla$.
We will denote by $d^H$ the exterior differential operator on
$\Omega\big(\pi_2^{-1}(\{s\}),{\mathcal E}_{M/B}\big)\cong{\mathcal E}_{M/S}$
associated to $\overline\nabla$.

The ``vertical'' differential operator $d^{\nabla_{\!\xi}}$ will be denoted by
$d^{M/B}$ on ${\mathcal E}_{M/B}$ and $d^{M/S}$ on ${\mathcal E}_{M/S}$. As was
remarked in the third paragraph of \S2.1 of \cite{MoiPartie2}, the difference $d^{M/S}-d^{M/B}-d^H$
(throw \eqref{E:isomabove}) is (the restrictions
to fibres of $\pi_2\circ\pi_1$ of) the operator $\iota_T\in\wedge^2(TB\!/\!S)^*\otimes{\text{End}}^{\text{odd}}
({\mathcal E}_{M/B})$ of \S2.1 of \cite{MoiPartie2}. It will be allways denoted $\iota_T$ here, and considered as an element of ${\rm{End}}{\mathcal E}_{M/S}$.
\subsection{Fibral Dirac operators:}
Endow $TM\!/\!B$ with some (riemannian) metric $g^{M/B}$ and $\xi$ with some hermitian metric
$h^\xi$. Take some riemannian metric $g^{B/S}$ on $TB\!/\!S$. Put on $TM\!/\!S$ the riemannian metric for which the decomposition $TM\!/\!S=TM\!/\!B\oplus T^H\!M$ is orthogonal and which co\"\i ncides with $g^{M/B}$ and $\pi_1^*g^{B/S}$ on either parts. The adjoints in ${\rm{End}}{\mathcal E}_{M/S}$
will be considered with respect to the $L^2$ scalar product on ${\mathcal E}_{M/S}$ obtained from $h^\xi$ and this riemannian metric (as in formula (38) of \cite{MoiPartie1}). These are not
the adjoints (neither usual nor special) considered on $\Omega\big(\pi_2^{-1}(\{s\}),{\mathcal E}_{M/B}\big)$
in \S2.2.4 of \cite{MoiPartie1}. For instance, the adjoint $\iota^*_T$ of $\iota_T$ here is not $T\wedge$ as it was in \S2.1 of \cite{MoiPartie2}.

Let $d^{M/B*}$ be the adjoint of $d^{M/B}$ with respect to
$g^{M/B}$ and $h^\xi$ as endomorphisms of ${\mathcal E}_{M/B}$, then $d^{M/B}$ and $d^{M/B*}$ are also adjoint as endomorphisms of ${\mathcal E}_{M/S}$ because of the choice of a riemannian submersion metric on $M$. Take some positive real number $\theta$, call $d^{H*}$ the adjoint of $d^H$ as endomorphism of ${\mathcal E}_{M/S}$ and put:
\begin{equation}\label{VFrontoni}
\begin{aligned}d^\theta=d^H+\frac1\theta d^{M/B}+\theta\iota_T\qquad&{\text{and}}\qquad
d^{\theta*}=d^{H*}+\frac1\theta d^{M/B*}+\theta \iota^*_T\\
{\mathcal D}^H=d^H+d^{H*}\qquad&{\text{and}}\qquad {\mathcal D}^V=d^{M/B}+d^{M/B*}\\
{\mathcal D}^\theta=d^\theta+d^{\theta*}&={\mathcal D}^H+\frac1\theta{\mathcal D}^V+\theta(\iota_T+\iota^*_T)\end{aligned}
\end{equation}
Let $N_V$ be the endomorphism of ${\mathcal E}_{M/B}$ defined as in \S2.5 of \cite{MoiPartie2}. $N_V$
multiplies $k$-degree vertical forms by $k$, ``vertical'' meaning forms along the fibres of $\pi_1$. Using $T^H\!M$, $N_V$ extends to
an operator on ${\mathcal E}_{M/S}$ throw the identification \eqref{E:isomabove}. Let $g_\theta$ be the riemannian metric on $M$ such that $TM\!/\!B$ and $T^H\!M$ are orthogonal and
which restricts to $g^{M/B}$ and $\frac1{\theta^2}\pi_1^*g^B$
on either part, the observation here is that ${\mathcal D}^\theta=\theta^{N_V}(d^{M/S}+d^{M/S*}_\theta)\theta^{-N_V}$
where $d^{M/S*}_\theta$ is the adjoint of $d^{M/S}$ with respect to $g_\theta$ and $h^\xi$ (see \cite{AlBerthomieuBismut} definition 5.5 for an analogous result in the holomorphic context). The riemanian submersion metric chosen here
simplifies considerably the form of ${\mathcal D}^\theta$ with
respect to the case of \cite{AlBerthomieuBismut} where such a choice is not allowed and forces more complicated conjugations than by $\theta^{N_V}$ (see \cite{AlBerthomieuBismut} \S5(a)).
\subsection{Introducing some intermediate suitable data:}
Consider some suitable data $(\eta^+,\eta^-,\psi)$ with respect to $\pi_1$ (and $\xi$ with $\nabla_{\!\xi}$ and $h^\xi$ and $g^{M/B}$).
$\eta^\pm$ are endowed with some hermitian metrics. Choose some connection $\nabla_{\!\eta}$ on $\eta^\pm$ (which respects either part) and the associated de Rham operator ${\mathcal D}^{\nabla_{\!\eta}}=d^{\nabla_{\!\eta}}+(d^{\nabla_{\!\eta}})^*$
on $\Omega(B/S,\eta^\pm)$.
Consider some function $\chi$ as in \S2.3 of \cite{MoiPartie2}. For $\theta\in(0,1]$, one puts
\begin{equation}\label{D+D+D}\begin{aligned}
{\mathcal D}^\theta_{\psi}&={\mathcal D}^\theta+{\mathcal D}^\eta+\frac{1-\chi(\theta)}\theta(\psi+\psi^*)\\&={\mathcal D}^H+{\mathcal D}^\eta+\frac1\theta{\mathcal D}^{V}_{(1-\chi(\theta))\psi}+\theta(\iota_T+\iota_T^*)\quad\!\in\!\quad{\rm{End}}^{\rm{odd}}\left({\mathcal E}_{M/S}\oplus\Omega
(B/S,\eta)\right)\end{aligned}\end{equation}
where ${\mathcal D}^V_{(1-\chi(\theta))\psi}$ is obtained from ${\mathcal D}^V$ and $(1-\chi(\theta))\psi$ as ${\mathcal D}^{\nabla_{\!\xi}}_\psi$ is from ${\mathcal D}^{\nabla_{\!\xi}}$
and $\psi$ in the first formula following (38) in \cite{MoiPartie1}.
Here $\psi$ is extended to forms on $B/S$ throw the isomorphism
\eqref{E:isomabove}. The choice of a riemannian submersion
metric on $TM$ ensures the compatibility of the adjunctions
of $\psi$ before and after extending it to forms on $B/S$.
This result \eqref{D+D+D} corresponds to \cite{AlBerthomieuBismut} proposition 5.9
with $\theta=\frac1T$ and with the extra term $\theta(\iota_T+\iota_T^*)$.

Associated to ${\mathcal D}^V_\psi$ is a double
decomposition
\begin{equation}\label{SarahRegardBleu}{\mathcal E}_{M/B}\oplus\eta^\pm={\rm{Ker}}{\mathcal D}_{\psi}^{V}
\oplus({\rm{Ker}}{\mathcal D}_{\psi}^{V})^\perp\end{equation}
which gives a double decomposition of infinite rank vector bundles on $S$:
\[{\mathcal E}_{M/S}\oplus\Omega
(B/S,\eta^\pm)=\Omega
\big(B/S,{\rm{Ker}}{\mathcal D}^{V}_\psi\big)\oplus
\Omega
\big(B/S,({\rm{Ker}}{\mathcal D}^{V}_\psi)^\perp\big)
\]
The choice of a riemannian submersion metric induces that this decomposition is orthogonal: let $p\colon{\mathcal E}_{M/S}\oplus\Omega
(B/S,\eta^\pm)\longrightarrow\Omega
\big(B/S,{\rm{Ker}}{\mathcal D}_{\psi}^V\big)$ be the orthogonal projection, it is the tensor product of the identity in $\Omega(B/S)$ and the orthogonal projection on the first factor of \eqref{SarahRegardBleu}. Put $p^\perp={\rm{Id}}-p$. For any positive $\theta$
one decomposes the operator ${\mathcal D}_{\psi}^\theta$
as a $2\times2$ matrix:
\[{\mathcal D}^{\theta}_\psi=\left(\begin{matrix}
p{\mathcal D}^{\theta}_{\psi}p&p{\mathcal D}^{\theta}_{\psi}p^\perp\\
p^\perp{\mathcal D}^{\theta}_{\psi}p&p^\perp{\mathcal D}^{\theta}_{\psi}p^\perp\end{matrix}\right)=:\left(
\begin{matrix}A^\theta_1&A^\theta_2\\A^\theta_3&A^\theta_4\end{matrix}\right)\]

As in \S3.3.3 of \cite{MoiPartie1} in the flat case and \S2.3 of \cite{MoiPartie2} in the general case, the vector bundle ${\rm{Ker}}{\mathcal D}^V_\psi$ is endowed with the restriction of the metric on ${\mathcal E}_{M/B}\oplus\eta^\pm$, and with the
connection $\nabla_{\!{\mathcal H}}$ obtained by projecting the connection on ${\mathcal E}_{M/B}\oplus\eta$ onto it (in fact $p(\overline\nabla\oplus\nabla_{\!\eta})p$, see \cite{AlBerthomieuBismut}
theorem 5.1 and formula (5.34)). Because of the compatibility of orthogonal projections, the exterior differential operator on $\Omega\big(B/S,{\rm{Ker}}{\mathcal D}_{\psi}^V\big)$
associated to this connection is $d^{\nabla_{\!{\mathcal H}}}=p(d^H\oplus d^{\nabla_{\!\eta}})p$. Put ${\mathcal D}^{\nabla_{\!{\mathcal H}}}=d^{\nabla_{\!{\mathcal H}}}+(d^{\nabla_{\!{\mathcal H}}})^*$, then clearly
$(d^{\nabla_{\!{\mathcal H}}})^*=p\big(d^{H*}\oplus (d^{\nabla_{\!\eta}})^*\big)p$ and
for any $\theta\leq\frac12$ (to ensure that $\chi(\theta)=0$), one has
\begin{equation}\label{convergop}
A^\theta_1={\mathcal D}^{\nabla_{\!{\mathcal H}}}+\theta p(\iota_T+\iota_T^*)p
\end{equation}
Of course $p(\iota_T+\iota_T^*)p$ is a bounded operator in the $L^2$-topology, and this remark with the above equation replaces here equation (5.35) of theorem 5.1 in \cite{AlBerthomieuBismut}.

In the same way, for $\theta\in[0,\frac12]$
\begin{equation}\label{nondiag}
\begin{aligned}
A_2^\theta&=p\big((d^H+d^{H*})\oplus (d^{\nabla_{\!\eta}}+(d^{\nabla_{\!\eta}})^*)\big)p^\perp+\theta p(\iota_T+\iota_T^*)p^\perp\\
{\text{and }}\qquad
A_3^\theta&=p^\perp\big((d^H+d^{H*})\oplus (d^{\nabla_{\!\eta}}+(d^{\nabla_{\!\eta}})^*)\big)p+\theta p^\perp(\iota_T+\iota_T^*)p\end{aligned}\end{equation}
are uniformly bounded operators in the $L^2$-topology.
This is because of the choice of the riemannian submersion metric
and is a simplification with respect to the corresponding result
proposition 5.18 of \cite{AlBerthomieuBismut}.
\subsection{Estimates on the operator $A^\theta_4$:}
First one wants to obtain results analoguous to \cite{AlBerthomieuBismut}
theorems 5.19 and 5.20.
There are three differences between the situation here and \cite{AlBerthomieuBismut}.
The absence of conjugation (by $C_T$ in \cite{AlBerthomieuBismut} definition 5.4 and (5.10)) due to the choice of a riemannian submersion metric is a simplification
and does not create any obstacle;
the presence here of the term $\theta (\iota_T+\iota_T^*)$ does not change these results
because of the fact that $\iota_T+\iota_T^*$ is a bounded operator in the
$L^2$-topology and because of its factor $\theta$; more seriously, the commutator
$[D^Z_\infty,D^H_\infty]$ in \cite{AlBerthomieuBismut} which corresponds to
$[{\mathcal D}^V,{\mathcal D}^H]$ in the notations here, is to be replaced by
\[[{\mathcal D}^V_\psi,{\mathcal D}^H+{\mathcal D}^\eta]=
[{\mathcal D}^V,{\mathcal D}^H]+[
\psi+\psi^*,{\mathcal D}^H+{\mathcal D}^\eta]\]
Of course the first term has the required majoration property
\cite{AlBerthomieuBismut} (5.67).
The operator 
$\psi+\psi^*$ is a fiberwise kernel operator (along the fibres of $\pi_1$),
and its kernel is smooth along the fibered double $M\times_B M$. Thus,
if ${\tt v}$ is a smooth vector field on $B$, the commutator
$\big[\psi+\psi^*,\nabla^{\wedge^\bullet T^*M/B\otimes\xi}_{\!{\tt v}^H}\oplus\nabla^\eta_{\!\tt v}\big]$
(where ${\tt v}^H$ is the horizontal lift of ${\tt v}$, a section of $T^H\!M$),
is a fiberwise kernel operator with globally smooth kernel. In particular, it is
bounded in $L^2$-topology, and so is the (super)commutator $[
\psi+\psi^*,{\mathcal D}^H+{\mathcal D}^\eta]$.
The estimate \cite{AlBerthomieuBismut} (5.67) then follows from
\cite{AlBerthomieuBismut} (5.61) (whose equivalent here holds true).

The conclusions of theorems 5.19
and 5.20 of \cite{AlBerthomieuBismut} remain thus valid here,
namely the existence of some constant $C$ such that for any $\theta\leq\frac12$ and any section $s$ of ${\mathcal E}_{M/S}\oplus\eta^\pm$
\begin{equation}\label{pub}
\Vert A^\theta_4(p^\perp s)\Vert_{L^2}\geq C\left(\Vert p^\perp s\Vert_{H^1}+\frac 1{\theta}\Vert p^\perp s\Vert_{L^2}\right)
\end{equation}
where $\Vert\ \, \Vert_{H^1}$ stands for the usual Sobolev $H^1$-norm.

Secondly, one needs some equivalent of \cite{AlBerthomieuBismut} proposition
5.22, particularly the estimate (5.71) contained in it.
But the proof here is in fact easier than in \cite{AlBerthomieuBismut}
because equation \eqref{nondiag} provides a simplification of the
corresponding proposition 5.18 in \cite{AlBerthomieuBismut}, the
extra term $\theta(\iota_T+\iota_T^*)$ is a uniformly bounded operator,
$(\psi+\psi^*)$ too, and $\frac1\theta(\psi+\psi^*)$ is part of $A^\theta_4$,
it does not disable the ellipticity of $A^\theta$
and it is taken into account in the obtained estimates: there exist constants
$c$, $C$ and $\theta_0$ such that for any $\theta\leq\theta_0$, $\lambda\in{\mathbb C}$ such that $\vert\lambda\vert\leq\frac c{2\theta}$ and any $s\in{\mathcal E}_{M/S}
\oplus\eta^\pm$,
\begin{equation}\label{at4}
\begin{aligned}
\Vert(\lambda-A^\theta_4)^{-1}p^\perp s\Vert_{L^2}&\leq C\theta\Vert p^\perp s\Vert_{L^2}\\
\Vert(\lambda-A^\theta_4)^{-1}p^\perp s\Vert_{H^1}&\leq C\Vert p^\perp s\Vert_{L^2}
\end{aligned}
\end{equation}
\subsection{Two-step direct image representative:}\label{VirginieFrontoni}
Choose now some suitable data $(\zeta^+,\zeta^-,\varphi)$ for ${\mathcal D}^{\nabla_{\!{\mathcal H}}}$, and extend $p$ and $p^\perp$ to ${\mathcal E}_{M/S}\oplus\zeta^\pm$ in the following way: $p$ induces the identity on $\zeta^\pm$ and $p^\perp$ induces the null map on $\zeta^\pm$.
Consider then 
\begin{align*}{\mathcal D}^\theta_{\psi,\varphi}&=\left(\begin{matrix}
p\big({\mathcal D}^{\theta}_{\psi}+(1-\chi(\theta))(\varphi+\varphi^*)\big)p&p{\mathcal D}^{\theta}_{\psi}p^\perp\\
p^\perp{\mathcal D}^{\theta}_{\psi}p&p^\perp{\mathcal D}^{\theta}_{\psi}p^\perp\end{matrix}\right)\\&=\left(
\begin{matrix}A^\theta_1+(1-\chi(\theta))(\varphi+\varphi^*) &A^\theta_2\\A^\theta_3&A^\theta_4\end{matrix}\right)\end{align*}
(It is not essential that the same function $\chi$ appears here and in \eqref{D+D+D}). Equation \eqref{convergop} obviously leads to
the following equality for $\theta\in[0,\frac12]$:
\begin{equation}\label{convergopphi}A^\theta_1+(1-\chi(\theta))(\varphi+\varphi^*)={\mathcal D}^{\nabla_{\!{\mathcal H}}}\!+\varphi+\varphi^*+\theta p(\iota_T+\iota_T^*)p
={\mathcal D}^{\nabla_{\!{\mathcal H}}}_\varphi\!+\theta p(\iota_T+\iota_T^*)p
\end{equation}
with the same remark (as after \eqref{convergop}) that $p(\iota_T+\iota_T^*)p$ is bounded.

Using this, the remark after \eqref{nondiag} above and \eqref{at4} instead of
\cite{AlBerthomieuBismut}(5.35), (5.49) and (5.71) respectively,
the analysis performed in \cite{AlBerthomieuBismut} \S\S5(d) and (g)
applies here. The only difference is that the following equivalent here of
the first line of \cite{AlBerthomieuBismut} (5.89) (for the usual norm of bounded operators in $L^2$-topology) is not true
\begin{equation}\label{estimeeverolee}\left\Vert\big(A^\theta_1+(1-\chi(\theta))(\varphi+\varphi^*)-{\mathcal D}^{\nabla_{\!{\mathcal H}}}_\varphi\big)(\lambda-{\mathcal D}^{\nabla_{\!{\mathcal H}}}_\varphi)^{-1}\right\Vert\leq C\theta^2(1+\vert\lambda\vert)\end{equation}
The set $U_T$ (or $U_{\frac1\theta}$) where $\lambda$ is supposed to lie, defined in \cite{AlBerthomieuBismut} (5.76), is such that
$\vert\lambda\vert\leq c_1T$ (or $\frac{c_1}\theta$) and $\Vert(\lambda-{\mathcal D}^{\nabla_{\!{\mathcal H}}}_\varphi)^{-1}\Vert\leq\frac{c_2}4$ for some constants $c_1$ and $c_2$. But only the following consequence of \eqref{estimeeverolee}
\[\left\Vert\big(A^\theta_1+(1-\chi(\theta))(\varphi+\varphi^*)-{\mathcal D}^{\nabla_{\!{\mathcal H}}}_\varphi\big)(\lambda-{\mathcal D}^{\nabla_{\!{\mathcal H}}}_\varphi)^{-1}\right\Vert\leq C\theta\]
is needed for establishing the equivalent of \cite{AlBerthomieuBismut}
(5.90). This last 
estimate can be obtained here directly from the remark after \eqref{convergopphi}
and the properties of $U_{\frac1\theta}$.

One obtains firstly the convergence of the resolvent of ${\mathcal D}^{\theta}_{\psi,\varphi}$ to any
great enough positive integral power $(\lambda-{\mathcal D}^{\theta}_{\psi,\varphi})^{-k}$
to $p(\lambda-{\mathcal D}^{\nabla_{\!{\mathcal H}}}_\varphi)^{-k}p$ in the sense of
the norm $\vert\!\vert A\vert\!\vert_1={\rm{tr}}(A^*A)^{\frac12}$
(\cite{AlBerthomieuBismut} theorem 5.28), and secondly the convergence of the spectral projector of ${\mathcal D}^{\theta}_{\psi,\varphi}$ with respect to eigenvalues of absolute value bounded
by some suitable positive constant $a$ to the orthogonal projector onto the
kernel of ${\mathcal D}^{\nabla_{\!{\mathcal H}}}_\varphi$ (\cite{AlBerthomieuBismut} equation (5.118)).


Thus the eigensections of ${\mathcal D}^{\theta}_{\psi,\varphi}$ corresponding to eigenvalues of absolute value bounded
by $a$ and the kernel of ${\mathcal D}^{\nabla_{\!{\mathcal H}}}_\varphi$ make vector bundles ${\mathcal G}^\pm$ over
$S\times[0,\varepsilon_2]$ for all sufficiently small (but positive) $\varepsilon_2$.
And for any $t\in[0,\varepsilon_2]$,
\[[{\mathcal G}^+\vert_{\{t\}}]-[{\mathcal G}^-\vert_{\{t\}}]
=\pi^{\rm{Eu}}_{2*}\big(\pi_{1*}^{\rm{Eu}}([\xi])+[\eta^+]-[\eta^-]\big)+[\zeta^+]-[\zeta^-]\ 
\in\ K^0_{\rm{top}}(S)\]
\begin{remark*}
Of course the convergence holds for the squared operators $\left({\mathcal D}^{\theta}_{\psi,\varphi}\right)^2$ and $\left({\mathcal D}^{\nabla_{\!{\mathcal H}}}_\varphi\right)^2$ too. In this case, the eigenspaces decompose
along the ${\mathbb Z}/2$-graduation decomposition of ${\mathcal E}_{M/S}\oplus\Omega
(B/S,\eta)\oplus\zeta$, and for a nonzero eigenvalue of $\left({\mathcal D}^{\theta}_{\psi,\varphi}\right)^2$, the nonsquared operator ${\mathcal D}^{\theta}_{\psi,\varphi}$ induces a bijection between the two parts
of the eigenspace. 
This proves that if ${\mathcal D}^{\nabla_{\!{\mathcal H}}}_\varphi$ is surjective on ${\mathcal E}^-_{B/S}\oplus\zeta^-$,
then ${\mathcal D}^{\theta}_{\psi,\varphi}$ has no nonvanishing little absolute valued eigenvalue for $\theta$ little enough,
so that
the kernel of ${\mathcal D}^{\theta}_{\psi,\varphi}$ converges to the kernel of
${\mathcal D}_\varphi^{\nabla_{\!{\mathcal H}}}$ and these kernels directly provide a vector bundle
on $S\times[0,\varepsilon_2]$.
\end{remark*}
\subsection{Link between one-step and two-step direct image representatives:}\label{llink}
The constructions of \S3.1 of \cite{MoiPartie1} (both the families analytic index map and the canonical link between different constructions) can be applied to ${\mathcal D}^\theta_\psi$ on any compact subset of $S\times(0,1]$. This is because
${\mathcal D}^\theta_\psi$ is a compact modification of the fibral elliptic
operator ${\mathcal D}^\theta\oplus{\mathcal D}^\eta$.


This does not work on $[0,1]$ because of the explosion of $A^\theta_4$ \eqref{pub}.

Choose $\varepsilon_1\in(0,\varepsilon_2)$ and suitable data $(\lambda^+,\lambda^-,\phi)$ for ${\mathcal D}^{\theta}_\psi$ with respect to the submersion $(\pi_2\circ\pi_1)\times{\rm{Id}}_{[\varepsilon_1,1]}\colon M\times[\varepsilon_1,1]\longrightarrow S\times[\varepsilon_1,1]$.
One then obtains two vector bundles ${\mathcal K}^\pm$ on
$S\times[\varepsilon_1,1]$ such that for any $\theta\in[\varepsilon_1,1]$
\[[{\mathcal K}^+_{S\times\{\theta\}}]-[{\mathcal K}^-_{S\times\{\theta\}}]=(\pi_2\circ\pi_1)^{\rm{Eu}}_*[\xi]+\pi_{2*}^{\rm{Eu}}([\eta^+]-[\eta^-])+[\lambda^+]-[\lambda^-]\in K^0_{\text{top}}(S)\]
This is obtained on $S\times\{1\}$
from the regular construction of \S3.1.1 of \cite{MoiPartie1}
and away from $\{1\}$ by parallel transport along $[\varepsilon_1,1]$.

Because $\varepsilon_1<\varepsilon_2$, and of the considerations above,
there is from \S3.1.2 of \cite{MoiPartie1} some canonical link between $({\mathcal G}^+
\oplus\zeta^-)
-({\mathcal G}^-\oplus\zeta^+)$ and $({\mathcal K}^+\oplus\lambda^-)-({\mathcal K}^-\oplus\lambda^+)$ over $B\times[\varepsilon_1,\varepsilon_2]$.
Represent $\pi_{2*}^{\rm{Eu}}([\eta^+]-[\eta^-])\in K^0_{\rm{top}}(S)$ by any difference
$[\mu^+]-[\mu^-]$ of vector bundles on $S$, one obtains (by adding the identity of $\mu^-$ and of
$\mu^+$) a link between $({\mathcal G}^+
\oplus\zeta^-\oplus\mu^-)
-({\mathcal G}^-\oplus\zeta^+\oplus\mu^-)$ which is a representative of
$\pi_{2*}^{\rm{Eu}}\big(\pi_{1*}^{\rm{Eu}}(\xi)\big)$, and $({\mathcal K}^+\oplus\lambda^-\oplus\mu^-)-({\mathcal K}^-\oplus\lambda^+\oplus\mu^+)$
which is a representative of
$(\pi_2\circ\pi_1)_*^{\rm{Eu}}(\xi)$
\subsection{Independence on the choices:}
Consider two systems of suitable data $(\zeta_1^+,\zeta_1^-,\varphi_1)$ and $(\zeta_2^+,\zeta_2^-,\varphi_2)$ for
${\mathcal D}^{\nabla_{\!{\mathcal H}}}$. Then
$({\rm{Ker}}
{\mathcal D}^{\nabla_{\!{\mathcal H}}+}_{\varphi_1}\oplus\zeta^-_1)-({\rm{Ker}}
{\mathcal D}^{\nabla_{\!{\mathcal H}}-}_{\varphi_1}\oplus\zeta^+_1)$ and $({\rm{Ker}}
{\mathcal D}^{\nabla_{\!{\mathcal H}}+}_{\varphi_2}\oplus\zeta^-_2)-({\rm{Ker}}
{\mathcal D}^{\nabla_{\!{\mathcal H}}-}_{\varphi_2}\oplus\zeta^+_2)$
are canonically linked from the construction of \S3.1.2 of \cite{MoiPartie1}, and this link
extends to some link between $({\mathcal G}^+\vert_{\{t\}}\oplus\zeta^-_1)-({\mathcal G}^-\vert_{\{t\}}\oplus\zeta^+_1)$ and $({\mathcal G}^+\vert_{\{t\}}\oplus\zeta^-_2)-({\mathcal G}^-\vert_{\{t\}}\oplus\zeta^+_2)$
for any $t$. But both this one and the one between $({\mathcal G}^+
\oplus\zeta^-)
-({\mathcal G}^-\oplus\zeta^+)$ and $({\mathcal K}^+\oplus\lambda^-)-({\mathcal K}^-\oplus\lambda^+)$ used just above are obtained from the construction of
\S3.1.2 of \cite{MoiPartie1}, they are thus compatible.

Now take two systems of suitable data $(\eta_1^+,\eta_1^-,\psi_1)$ and $(\eta_2^+,\eta_2^-,\psi_2)$ with respect to $\pi_1$ (and $\xi$ with $\nabla_{\!\xi}$ and $h^\xi$ and $g^{M/B}$).
There is a link between $\big({\rm{Ker}}
{\mathcal D}_{M/B}^{\psi_1+}\oplus\eta_1^-\big)-\big(\eta_1^+\oplus{\rm{Ker}}
{\mathcal D}_{M/B}^{\psi_1-}\big)$ and $\big({\rm{Ker}}
{\mathcal D}_{M/B}^{\psi_2+}\oplus\eta^-_2\big)-\big(\eta_2^+\oplus{\rm{Ker}}
{\mathcal D}_{M/B}^{\psi_2-}\big)$ as constructed in \S3.1.2 of \cite{MoiPartie1}.
This link is obtained by constructing a families index map for
a submersion of the form $\pi_1\times{\rm{Id}}_{[0,1]}\colon M\times[0,1]
\longrightarrow B\times[0,1]$. This construction can be extended to the case of a double
submersion in the following form $M\times[0,1]\overset{\pi_1\times{\rm{Id}}_{[0,1]}}
\longrightarrow B\times[0,1]\overset{\pi_2\times{\rm{Id}}_{[0,1]}}
\longrightarrow S\times[0,1]$, and the compatibility of canonical
links either for linked data or for one and for two submersions follows.

\section{Flat $K$-theory:}\label{flat}
\subsection{Leray spectral sequence:}\label{Nadia}
Consider some flat vector bundle $(F,\nabla_{\!F})$ on $M$, let $G^\bullet
=H^\bullet(M/B,F\vert_{M/B})$ be the graded flat bundle on $B$ of the $F$-valued
cohomology of the fibres of $\pi_1$, (with flat connections $\nabla_{\!G^\bullet}$) and $H^\bullet
=H^\bullet(M/S,F\vert_{M/S})$ be the graded flat bundle on $S$ of the $F$-valued
cohomology of the fibres of $\pi_2\circ\pi_1$ (with flat connections $\nabla_{\!H^\bullet}$; of course in the notation of section 3.3 of \cite{MoiPartie1} $\pi_{1!}^+F=
\mathop\oplus\limits_{i{\text{ even}}}G^i$, $\pi_{1!}^-F=
\mathop\oplus\limits_{i{\text{ odd}}}G^i$ and similarly for $H^\bullet$
and $(\pi_2\circ\pi_1)_!F$).

The vertical $F$-valued de Rham complex $\Omega^\bullet(M/S,F)$ along $M/S$ is filtrated
by the horizontal degree: for any $p$, $F^p\Omega^\bullet(M/S,F)$ consists
of $F$-valued differential forms whose interior product with more than $p$
elements of $TM/B$ vanishes.
Thus $H^\bullet$ is also filtrated from this filtration: $F^pH^\bullet$
consists of classes which can be represented by some element in 
$F^p\Omega^\bullet(M/S,F)$. This filtration is compatible with the flat
connections of $H^\bullet$, so that for any $p$ and $k$,
\begin{equation}\label{Christine}0\longrightarrow F^{p+1}H^k\longrightarrow F^pH^k
\longrightarrow F^pH^k/F^{p+1}H^k\longrightarrow0\end{equation}
is an exact sequence in the category of flat bundles. The corresponding flat connections will be respectively denoted by $\nabla_{\!F^{p+1}H^k}$, $\nabla_{\!F^{p}H^k}$ and $\nabla_{\!p/k}$.

It is proved in \cite{Ma} Proposition 3.1 that the associated spectral sequence gives rise to
flat vector bundles $(E^{p,q}_r,\nabla^{p,q}_r)$ on $S$ with flat (parallel)
spectral sequence morphisms $d_r\colon E^{p,q}_r\longrightarrow E^{p-r,q+r+1}_r$
(and $E_{r+1}^{p,q}={\rm{Ker}}d_r\vert_{E^{p,q}_r}\big/({\rm{Im}}d_r\cap E^{p,q}_r)$).

It is a classical fact (see \cite{Ma} theorem 2.1) that this spectral sequence is isomorphic to the
Leray spectral sequence, and thus
$E^{p,q}_2\cong H^p(B/S,G^q)$ while $E^{p,q}_r\cong
F^pH^{p+q}\big/F^{p+1}H^{p+q}$ for all sufficiently great $r$.

Put $E^+_r=\mathop\oplus\limits_{p+q{\text{ even}}}E^{p,q}_r$ and
$E^-_r=\mathop\oplus\limits_{p+q{\text{ odd}}}E^{p,q}_r$, and denote their direct sum (flat) connections by $\nabla^+_{\!r}$ and $\nabla^-_{\!r}$. The construction of
formulae (5) and (6) of \cite{MoiPartie1} for the complex associated to $d_r$ produces for any $r\geq 2$
a link between $E^+_r-E^-_r$ and $E^+_{r+1}-E^-_{r+1}$. Decomposing this complex in
direct sums of short exact sequences in the classical way proves that this link
is a sum of links of the form appearing in relation $(iii)$ of definition
1 in \cite{MoiPartie1}, so that the element $[E^+_r,\nabla^+_{\!r}]
-[E^-_r,\nabla^-_{\!r}]\in K^0_{\rm{flat}}(S)$ is independent of $r$.

For $r=2$, this is nothing but $\pi_{2!}([G^+,\nabla_{\!G^+}]-[G^-,\nabla_{\!G^-}])
=\pi_{2!}\big(\pi_{1!}[F,\nabla_{\!F}]\big)$.

In the other hand, it follows from \eqref{Christine} that the element
\[[F^pH^\bullet,\nabla_{\!F^pH^\bullet}]+\sum_{i=0}^{p-1}[F^iH^\bullet/F^{i+1}H^\bullet,\nabla_{\!i/\bullet}]
\in K^0_{\rm{flat}}(S)\]
is independent of $p$. For $p=0$, it equals $[H^\bullet,\nabla_{\!H}]=
(\pi_2\circ\pi_1)_![F,\nabla_{\!F}]$, while for sufficienly great $p$ and $r$, it equals $[E^+_r,\nabla^+_{\!r}]
-[E^-_r,\nabla^-_{\!r}]$. Thus
\begin{proposition*}
$\pi_{2!}\circ\pi_{1!}=(\pi_2\circ\pi_1)_!\ \colon\ K^0_{\rm{flat}}(M)\longrightarrow K^0_{\rm{flat}}(S)$.
\end{proposition*}
\subsection{Compatibility of topological and sheaf theoretic links:}
Two links between $E_2^+-E_2^-\cong\pi_{2!}\big(\pi_{1!}[F,\nabla_{\!F}]\big)$
and $H^+-H^-\cong(\pi_2\circ\pi_1)_![F,\nabla_{\!F}]$ are now available: the link $[\ell_{\rm{top}}]$
constructed in subsection \ref{llink} and the sheaf theoretic one $[\ell_{\rm{flat}}]$
constructed just above from the Leray spectral sequence and the filtration
of $H^\bullet$ \eqref{Christine}.
\begin{proposition}\label{flatop}$\qquad[\ell_{\rm{top}}]=[\ell_{\rm{flat}}]$
\end{proposition}
\begin{proof}
We will use the Hodge theoretic version of the Leray spectral
sequence (at $\theta=0$). Such a theory was studied by various authors
in various contexts
\cite{MazzeoMelrose} \cite{ThesedeDai} \cite{AlBerthomieuBismut} \cite{Ma},
the version corresponding to the situation here in explained in \cite{Ma}
\S2 and \S3.
It can be summarized as follows: $E_0$ is nothing but ${\mathcal E}_{M/S}$ (see \eqref{E:isomabove}) as global infinite rank vector bundle over $S$.
Then there exists a nested sequence of vector subbundles $\widetilde E_r$
of $E_0=\widetilde E_0$ which are for all $r\geq2$ of finite rank and endowed with canonical flat
connections $\widetilde\nabla_{\!r}$. This sequence stabilizes
for sufficiently great $r$. For any $r$, there is some canonical isomorphism
$E_r\cong\widetilde E_r$ with the corresponding term of the Leray spectral sequence, for $r\geq2$ it makes $\nabla_{\!r}$ and $\widetilde\nabla_{\!r}$ correspond to each other. All the $\widetilde E_r$ are naturally endowed
with the restriction of the $L^2$ hermitian inner product on ${\mathcal E}_{M/S}$
(which needs here to be obtained from some riemannian submersion metric).
Finally for any $r$, let $\widetilde d^*_r$ be the adjoint of the bundle endomorphism $\widetilde d_r$ corresponding
to the operator $d_r$ of the spectral sequence, and define $\widetilde{\mathcal D}_r=\widetilde d_r+\widetilde d^*_r$, then $\widetilde E_{r+1}:={\rm{Ker}}\widetilde{\mathcal D}_r$.

For $r=0$, $d_0\cong d^{M/B}$, so that $E_1$ identifies throw fibral Hodge theory with $\widetilde E_1=\Omega(B\!/\!S,{\rm{Ker}}{\mathcal D}^{V})
\cong\Omega(B\!/\!S,G^\bullet)$ in the notations of the preceding paragraph. Thus $\widetilde E_1$ identifies with vertical differential forms with values in $\pi_{1!}F$, where ``vertical'' is to understand with respect to the fibration $\pi_2$.
Let $\widetilde p_1$ be the orthogonal projection of $E_0$ onto
$\widetilde E_1$; then $\widetilde d_1=\widetilde p_1 d^H$ acting on $\widetilde E_1$, so that $\widetilde E_2={\rm{Ker}}
(p_1{\mathcal D}^H\vert_{\widetilde E_1})$
identifies with fibral harmonic $G^\bullet$-valued differential forms, hence
with $\pi_{2!}(\pi_{1!}F)$.

For any $r\geq2$, $\widetilde E_r$ can be described as follows (\cite{Ma}
Proposition 2.1):
\begin{equation}\label{Virginie} 
\begin{aligned}\widetilde E_r\, =\, \big\{&s_0\in{\mathcal E}_{M/S}\ {\text{ such that there exists }}s_1,s_2,\ldots,s_{r-1}\in{\mathcal E}_{M/S}{\text{ verifying}}\\
&\ \ {\mathcal D}^Vs_0=0,\ {\mathcal D}^Hs_0+{\mathcal D}^Vs_1=0\ {\text{ and }}\\
&\ \ \ \  (\iota_T+\iota^*_T)s_{i-2}+{\mathcal D}^Hs_{i-1}+{\mathcal D}^Vs_i=0\ {\text{ for any }}2\leq i\leq r-1\big\}
\end{aligned}
\end{equation}
Then in this description $\widetilde{\mathcal D}_r s_0=\widetilde p_r((\iota_T+\iota^*_T)s_{r-2}+{\mathcal D}^Hs_{r-1})$, where $\widetilde p_r$
is the orthogonal projection of ${\mathcal E}_{M/S}$ onto $\widetilde E_r$.
One can then prove along the same lines as in \cite{AlBerthomieuBismut}
\S VI (a) (especially formulae (6.13) and (6.15)) that
\begin{align*}
d^{M/B}s_0=0,\ &d^Hs_0+d^{M/B}s_1=0,\\ 
\iota_Ts_{i-2}&+d^Hs_{i-1}+d^{M/B}s_i=0\ {\text{ for any }}2\leq i\leq r-1\\
{\text{ and }}\ &\widetilde d_rs_0=\widetilde p_r(\iota_T s_{r-2}+d^Hs_{r-1})
\end{align*}

Use now the convergence of the resolvent $\big(\lambda-(\frac1\theta\big)^{r-1}
{\mathcal D}^\theta\big)^{-1}$ (here both $\eta$ and $\psi$ vanish) to
$\widetilde p_r(\lambda-\widetilde{\mathcal D}_r)^{-1}\widetilde p_r$
(\cite{Ma} Theorem 2.2) for sufficiently large $r$.
One can deduce that the orthogonal projection $p_\theta$ of ${\mathcal E}_{M/S}$
onto ${\rm{Ker}}{\mathcal D}^\theta$ converges at $\theta=0$ to
$\widetilde p_r$.
In other words ${\rm{Ker}}{\mathcal D}^\theta$ is the restriction to $S\times(0,1]$ of some vector bundle on
$S\times[0,1]$ whose restriction to $S\times\{0\}$ is $\widetilde E_\infty$.
There is a bigrading on ${\mathcal E}_{M/S}$, from \eqref{E:isomabove}
according to horizontal (i.e. corresponding to $\Omega^\bullet$) and vertical
(corresponding to the grading of ${\mathcal E}_{M/B}$) degrees.
$\widetilde E_\infty$ decomposes with respect to this bigrading
\cite{AlBerthomieuBismut} theorem 6.1. Consider some $s_0\in\widetilde
E_\infty^{p,q}$ and call $s_i^{p+i,q-i}$ for any $i$ the corresponding component of the $s_i$ introduced in \eqref{Virginie}.
The above description of $\widetilde d_r$ proves that for any sufficiently large $r$ the differential form $s_0+s_1^{p+1,q-1}+\ldots s_r^{p+r,q-r}$ is closed. According to the scaling appearing in \eqref{VFrontoni}
the section $p_\theta(s_0+\theta s_1^{p+1,q-1}+\theta^2 s_2^{p+2,q-2}\ldots \theta^rs_r^{p+r,q-r})$ is the rescaled harmonic form corresponding to some fixed cohomology class. Its obvious convergence to $s_0$ at $\theta=0$ proves that the isomorphism between ${\rm{Ker}}{\mathcal D}^1$ ad $\widetilde E_\infty$ provided by the parallel transport along $[0,1]$
exactly corresponds to the isomorphism
$[H^\bullet,\nabla_{\!H}]\cong[E^\bullet_r,\nabla_{\!r}]$ obtained at the end
of \S\ref{Nadia} from the exact sequences \eqref{Christine}.

In the notations of paragraph \ref{llink}, ${\mathcal K}={\rm{Ker}}{\mathcal D}^\theta$, so that one can take here $\varepsilon_1=0$.
The convergence of  the resolvent $\big(\lambda-(\frac1\theta\big)^{r-1}
{\mathcal D}^\theta\big)^{-1}$ to
$\widetilde p_r(\lambda-\widetilde{\mathcal D}_r)^{-1}\widetilde p_r$
(\cite{Ma} Theorem 2.2) for any $r$ gives the following description of
the vector bundle ${\mathcal G}$ over $S\times[0,\varepsilon_1]$: its restriction to $S\times\{\theta\}$
is the direct sum of eigenspaces of ${\mathcal D}_\theta$
corresponding to ``little'' modulus eigenvalues while its restriction to
$S\times\{0\}$ is the direct sum of the $\widetilde E_r$, each $\widetilde E_r$
corresponding to eigenspaces associated to eigenvalues of order less than or equal to $\theta^{r-1}$. For any positive $\theta$, $({\mathcal G},d^\theta)$ form a complex whose cohomology is ${\mathcal K}$. Accordingly, the canonical link between ${\mathcal G}^+-{\mathcal G}^-$ and ${\mathcal K}^+-{\mathcal K}^-$ is provided either by formula (5)
(using $d^\theta$) or (7) (using ${\mathcal D}^\theta$) of \cite{MoiPartie1}.
In the same way, the canonical link berween $E^+_{r-1}-E^-_{r-1}$ and $E^+_r-E^-_r$ is obtained from $d_r$ by formula (5) of \cite{MoiPartie1} and it corresponds to the canonical link between $\widetilde E^+_{r-1}-\widetilde E^-_{r-1}$ and $\widetilde E^+_r-\widetilde E^-_r$ obtained from $\widetilde{\mathcal D}_r$
by formula (7) of \cite{MoiPartie1}. The convergence of the resolvents also prove that the operator $(\frac1\theta)^{r-1}{\mathcal D}^\theta$ on the suitable eigensubspace
converges to $\widetilde {\mathcal D}_r$, and accordingly for $(\frac1\theta)^{r-1}d^\theta$ and $\widetilde d_r$. Thus, the canonical (topological) link
between ${\mathcal G}^+-{\mathcal G}^-$ and ${\mathcal K}^+-{\mathcal K}^-$
converges to the canonical (sheaf theoretical) link between $E_2^+-E_2^-$ and $E^+_\infty-E^-_\infty$ considered in the preceding paragraph.

Combining this with the considerations on the limit up to $\theta=0$ of ${\rm{Ker}}{\mathcal D}^\theta$ proves the proposition.
\end{proof}
\section{Relative $K$-theory:}\label{relat}
\begin{theorem*}
$\pi_{2*}\circ\pi_{1*}=(\pi_2\circ\pi_1)_*\in{\rm{Hom}}\big(K^0_{\rm{rel}}(M),
K^0_{\rm{rel}}(S)\big)$
\end{theorem*}
\begin{proof}
Take some $(E,\nabla_{\!E},F,\nabla_{\!F},f)\in K^0_{\text{rel}}(M)$,
then $\pi_{2*}\circ\pi_{1*}(E,\nabla_{\!E},F,\nabla_{\!F},f)$ is of the form
$\big(\pi_{2!}\circ\pi_{1!}(E,\nabla_{\!E}),\pi_{2!}\circ\pi_{1!}(F,\nabla_{\!F}),
[\ell]\big)$ while $(\pi_2\circ\pi_1)_*(E,\nabla_{\!E},F,\nabla_{\!F},f)$ is of
the form $\big((\pi_2\circ\pi_1)_!(E,\nabla_{\!E}),(\pi_2\circ\pi_1)_!(F,\nabla_{\!F}),
[\ell']\big)$ where $[\ell]$ and $[\ell']$ are suitable classes of links.

Consider the pull-back $\widetilde E$ of $E$ to $M\times[0,1]$ with some connection
$\widetilde \nabla$ whose restrictions on $M\times\{0\}$ and $M\times\{1\}$
respectively equal $\nabla_{\!E}$ and $f^*\nabla_{\!F}$. There is a canonical (topological) class of link $[\widetilde\ell]$ between one-step and two-step
direct images of $\widetilde E$ whose restrictions to $M\times\{0\}$
and $M\times\{1\}$ coincide with $[\ell^E_{\rm{top}}]$ and $[\ell^F_{\rm{top}}]$ (with obvious notations from the preceding subsection,
this is because of the naturality of $[\ell_{\rm{top}}]$). Now $[\ell]$ and
$[\ell']$ both correspond to the parallel transport along $[0,1]$ for
one step and two step direct image respectively. It follows that
\begin{align*}
\pi_{2*}&\circ\pi_{1*}(E,\nabla_{\!E},F,\nabla_{\!F},f)-(\pi_2\circ\pi_1)_*(E,\nabla_{\!E},F,\nabla_{\!F},f)=\\&\qquad\quad=\big(\pi_{2!}\circ\pi_{1!}(E,\nabla_{\!E}),(\pi_2\circ\pi_1)_!(E,\nabla_{\!E}),
[\ell_{\rm{top}}^E]\big)\\&\qquad\qquad\qquad\qquad\quad\qquad\qquad-\big(\pi_{2!}\circ\pi_{1!}(F,\nabla_{\!F}),(\pi_2\circ\pi_1)_!(F,\nabla_{\!F}),
[\ell^F_{\rm{top}}]\big)\end{align*}
But in  both cases $[\ell_{\rm{top}}]=[\ell_{\rm{flat}}]$ and $\ell_{\rm{flat}}$ is only obtained from exact sequences in the category of flat vector bundles (either from the flat complex associated to $d_r$ or from \eqref{Christine}).
Thus both terms in the right hand side of the above equation vanish, and this proves the theorem.
\end{proof}
\section{Multiplicative and free multiplicative $K$-theory:}\label{mult}
Consider the vector bundles $\xi$ on $M$, $F^+$ and $F^-$ on
$B$ and $G^+$ and $G^-$ on $S$ (with connections $\nabla_{\!\xi}$, $\nabla_{\!F^+}$, $\nabla_{\!F^-}$,
$\nabla_{\!G^+}$ and $\nabla_{\!G^-}$) such that
\[[F^+]-[F^-]=\pi_{1*}^{\rm{Eu}}[\xi]\in K^0_{\rm{top}}(B)\quad{\text{ and }}\quad
[G^+]-[G^-]=(\pi_2\circ\pi_1)_*^{\rm{Eu}}[\xi]\in K^0_{\rm{top}}(S)\]
Choose some smooth supplementary subbundle $T^H\!M\!/\!S$ of $TM\!/\!S$
in $TM$, such that $T^H\!M\!/\!S\cap TM\!/\!B=T^H\!M$; then $\pi_{1*}T^H\!M\!/\!S$ is a smoth supplementary subbundle of $TB\!/\!S$
in $TB$. One can define connections $\nabla_{\!TM/B}$, $\nabla_{\!TM/S}$
and $\nabla_{\!TB/S}$ on $TM\!/\!B$, $TM\!/\!S$ and $TB\!/\!S$ as at
the beginning of the proof lemma 12 of \cite{MoiPartie1} from the choices of horizontal subspaces $T^H\!M$, $T^H\!M\!/\!S$ and $\pi_{1*}T^H\!M\!/\!S$ respectively. Let $[\ell_F]$
and $[\ell_G]$ be equivalence classes of links between either $F^+-F^-$ or $G^+-G^-$ and bundles provided with the families analytic index construction
(as in definition 7 of \cite{MoiPartie2}), and denote
$\tau_1=\tau(\nabla_{\!\xi}, \nabla_{\!TM/B}, \nabla_{\!F^+},\nabla_{\!F^-},[\ell_F])$ and
$\tau_{12}=\tau(\nabla_{\!\xi}, \nabla_{\!TM/S}, \nabla_{\!G^+},\nabla_{\!G^-},[\ell_G])$.
Thus
\begin{align*}\pi_{1!}^{\rm{Eu}}(\xi,\nabla_{\!\xi},\alpha)=\left(F^+,\nabla_{\!F^+},\int_{M/B}
e(\nabla_{\!TM/B})\alpha\right)-(F^-,\nabla_{\!F^-},
\tau_1)\in\widehat K_{\rm{ch}}(B)\\
{\text{and }}\quad(\pi_2\circ\pi_1)_!^{\rm{Eu}}(\xi,\nabla_{\!\xi},\alpha)=\qquad\qquad\qquad\qquad\qquad\qquad\qquad\quad\qquad\qquad\qquad\\=\left(G^+,\nabla_{\!G^+},\int_{M/S}
e(\nabla_{\!TM/S})\alpha\right)-(G^-,\nabla_{\!G^-},
\tau_{12})\in\widehat K_{\rm{ch}}(S)
\end{align*}
Take vector bundles $H^{++}$, $H^{+-}$, $H^{-+}$ and $H^{--}$ on $S$ (with connections $\nabla_{\!++}$, $\nabla_{\!+-}$, $\nabla_{\!-+}$ and $\nabla_{\!--}$) such that $\pi_{2*}[F^\pm]=[H^{\pm+}]-[H^{\pm-}]\in K^0_{\rm{top}}(S)$. Consider some equivalence class of links $[\ell_+]$ and $[\ell_-]$ between $H^{\pm+}-H^{\pm-}$ and bundles provided with the families analytic index construction and denote by
$\tau_\pm$ the forms $\tau(\nabla_{\!F^\pm},\nabla_{\!TB/S},\nabla_{\!\pm+},\nabla_{\!\pm-},[\ell_\pm])$. Then
\begin{align*}
\pi_{2!}^{\rm{Eu}}\big(\pi_{1!}^{\rm{Eu}}(\xi,\nabla_{\!\xi},\alpha)\big)=
\left(H^{++},\nabla_{\!++},\int_{B/S}e(\nabla_{\!TB/S})\int_{M/B}e(\nabla_{\!TM/B})\alpha\right)\qquad\qquad\\-(H^{+-},\nabla_{\!+-},\tau_+)
-\left(H^{-+},\nabla_{\!-+},\int_{B/S}e(\nabla_{\!TB/S})\tau_1\right)
+(H^{--},\nabla_{\!--},\tau_-)
\end{align*}
Now $G^+-G^-$ and $(H^{++}\oplus H^{--})-(H^{+-}\oplus H^{-+})$ are linked throw
$[\ell_G]$, $[\ell_+]$, $[\ell_-]$ and the construction of \S\ref{llink}. Call
$[\ell_{\rm{top}}]$ the resulting link and $\widetilde{\rm{ch}}([\ell_{\rm{top}}])$ the associated Chern-Simons form as in
\S2.2.5 of \cite{MoiPartie1}, then
\begin{align*}
\pi_{2!}^{\rm{Eu}}\big(\pi_{1!}^{\rm{Eu}}(\xi,\nabla_{\!\xi},\alpha)\big)=
\left(G^+,\nabla_{\!G^+},\int_{M/S}\pi_1^*\big(e(\nabla_{\!TB/S})\big) e(\nabla_{\!TM/B})\alpha\right)\qquad\qquad\quad\\
-\left(G^-,\nabla_{\!G^-},\tau_+-\tau_--\widetilde{\rm{ch}}([\ell_{\rm{top}}])+\int_{B/S}e(\nabla_{\!TB/S})\tau_1\right)
\end{align*}
Choose any supplementary subbundle of $T^H\!M$ in $T^H\!M\!/\!S$, it then identifies with $\pi_1^*TB\!/\!S$ and is endowed with the connection $\pi_1^*\nabla_{\!TB/S}$.
Denote by $\widetilde e_{M/B/S}$ the form $\widetilde e(\nabla_{\!TM/S},\nabla_{\!TM/B}\oplus\pi_1^*\nabla_{!TB/S})$
defined in \S2.2.2 of \cite{MoiPartie1}, then the following form
\[\widetilde e_{M/B/S}d\alpha
+\big(e(\nabla_{\!TM/S})-\pi_1^*\big(e(\nabla_{\!TB/S})\big) e(\nabla_{\!TM/B})\big)\alpha\]
is exact so that in $\widehat K_{\rm{ch}}(S)$:
\begin{align*}
\pi_{2!}^{\rm{Eu}}\big(\pi_{1!}^{\rm{Eu}}(\xi,\nabla_{\!\xi},\alpha)\big)=\left(G^+,\nabla_{\!G^+},\int_{M/S}
e(\nabla_{\!TM/S})\alpha\right)-(G^-,\nabla_{\!G^-},
\widetilde\tau_{12})\\
{\text{with }}\qquad\widetilde\tau_{12}=\tau_+-\tau_--\widetilde{\rm{ch}}([\ell_{\rm{top}}])+\int_{B/S}e(\nabla_{\!TB/S})\tau_1-\int_{M/S}\widetilde e_{M/B/S}d\alpha
\end{align*}
\begin{theorem*}
The restriction to $MK^0(M)$ of $\ \pi_{2!}^{\rm{Eu}}\circ\pi_{1!}^{\rm{Eu}}$ and $(\pi_2\circ\pi_1)_!^{\rm{Eu}}$
coincide.
\end{theorem*}
\begin{proof}
For any $(\xi,\nabla_{\!\xi},\alpha)\in MK_0(M)$, one has $d\alpha={\rm{ch}}
(\nabla_{\!\xi})-{\rm{rk}}\xi$ but for degree reasons $\int_{M/S}\widetilde e_{M/B/S}$ vanishes (the degree of this form equals ${\rm{dim}}M-{\rm{dim}}S-1$).
So, the problem is reduced to the equality of $\tau_{12}$ and of
\[\overset\approx\tau_{12}=\tau_+-\tau_--\widetilde{\rm{ch}}([\ell_{\rm{top}}])+\int_{B/S}e(\nabla_{\!TB/S})\tau_1+\int_{M/S}\widetilde e_{M/B/S}{\rm{ch}}
(\nabla_{\!\xi})\]
for any $(\xi,\nabla_{\!\xi},\alpha)\in MK^0(M)$.

It is easily verified that $\overset\approx\tau_{12}$ 
is additive in the sense of lemma 8 of \cite{MoiPartie2}, is functorial by
pullbacks over fibered products (with double fibration structure!);
a direct calculation proves that it verifies the same transgression formula as $\tau_{12}$ (see
formula (12) of \cite{MoiPartie2}).
Moreover, in the case of a flat bundle $(\xi,\nabla_{\!\xi})$, $F^\pm$ here correspond to $G^\pm$ in \S\ref{Nadia}, $G^\pm$ here
correspond to $H^\pm$ of \S\ref{Nadia}, and $H^{\pm\pm}$ here correspond to $E_2^{\pm\pm}$ of \S\ref{Nadia}, and in any case, the suitable data
are taken trivial because all bundles are flat. Thus all the forms $\tau_+$,
$\tau_-$ and $\tau_1$ vanish (see lemma 8 of \cite{MoiPartie2}), ${\rm{ch}}(\nabla_{\!\xi})={\rm{rk}}\xi$ so that
the integral involving $\widetilde e_{M/B/S}$ vanishes, and $\widetilde{\rm{ch}}
([\ell_{\rm{top}}])$ also vanishes, because of proposition \ref{flatop} and lemma 6 of \cite{MoiPartie1} (and the description of $[\ell_{\rm{flat}}]$
of proposition \ref{flatop} as obtained fom exact sequences of flat bundles).

The coincidence of $\tau_{12}$ and $\overset\approx\tau_{12}$ for elements of $MK_0(M)$ is then obtained from the caracterisation of theorem of \S2.7.2 in \cite{MoiPartie2}.
\end{proof}
It is likely that $\overset\approx\tau_{12}=\tau_{12}$ in any case, so that
\[\pi_{2!}^{\rm{Eu}}\big(\pi_{1!}^{\rm{Eu}}(\xi,\nabla_{\!\xi},\alpha)\big)
=\big(0,0,\overset\approx\tau_{12}-\tau_{12}\big)=\left(0,0,\int_{M/S}\widetilde e_{M/B/S}\widehat{\rm{ch}}(\xi,\nabla_{\!\xi},\alpha)\right)\]
for any $(\xi,\nabla_{\!\xi},\alpha)\in\widehat K_{\rm{ch}}(M)$
(where $\widehat{\rm{ch}}(\xi,\nabla_{\!\xi},\alpha)={\rm{ch}}(\nabla_{\!\xi})
-d\alpha$ vanishes on $MK^0(M)$). This formula would be compatible with the preceding theorem and with the anomaly formulas (13) and (14) of
\cite{MoiPartie2}.


\begin{thebibliography}{9}
\bibitem{MoiOberwolfach}
A. Berthomieu: {\emph{Direct Images for Relative and Multiplicative $K$-Theories}}, Oberwolfach Reports, Vol. {\bf{3}}, Nr 1 (2006) pp. 758-760.
\bibitem{MoiPartie1}
A. Berthomieu: {\emph{Direct image for multiplicative and relative $K$-theories from transgression of the families index theorem, part 1.}} preprint at {\tt{arXiv:math.DG/0611281}}
\bibitem{MoiPartie2}
A. Berthomieu: {\emph{Direct image for multiplicative and relative $K$-theories from transgression of the families index theorem, part 2.}} preprint at {\tt{arXiv:math.DG/0703916}}
\bibitem{AlBerthomieuBismut}
A. Berthomieu J.-M. Bismut: {\emph{Quillen metrics and higher analytic torsion forms}}, J. reine angew. Math. {\bf{457}} (1994) 85-184.
\bibitem{BunkeSchick} U. Bunke, Th. Schick: {\emph{Smooth $K$-theory}},
preprint at {\tt{arXiv:0707.0046}} 
\bibitem{ThesedeDai} X. Dai: {\emph{Adiabatic limits, nonmultiplicativity of signature and Leray spectral sequences}}, J. A. M. S. {\bf{4}} (1991) 265-321.
\bibitem{Ma} X. Ma: {\emph{Functoriality of real analytic torsion forms}}, Isra\"el
J. of Math. {\bf{131}} (2002), pp. 1-50.
\bibitem{MazzeoMelrose}R. Mazzeo, R. Melrose: {\emph{The adiabatic limit, Hodge cohomology and Leray's spectral sequence for a fibration}},
J. of Diff. Geom. {\bf{31}} (1990) 185-213.
\end{thebibliography}
\end{document}